\documentclass[12 pt]{amsart}

\usepackage{amsthm}
\usepackage{amsmath}
\usepackage{amssymb}
\usepackage{latexsym}
\usepackage{url}
\usepackage{graphicx}
\usepackage{bmpsize}
\usepackage{enumitem}

\author{Felix Goldberg}
\address{Caesarea-Rothschild Institute, University of Haifa, Haifa, Israel}
\email{felix.goldberg@gmail.com}

\title{Spectral radius minus average degree: a better bound}
\date{July 15, 2013}

\newtheorem{thm}{Theorem}[]
\newtheorem{lem}{Lemma}[]
\newtheorem{cor}{Corollary}[]
\newtheorem{defin}{Definition}[]
\newtheorem{expl}{Example}[]
\newtheorem{qstn}{Question}[]

\DeclareMathOperator{\var}{var}

\begin{document}

\begin{abstract}
Collatz and Sinogowitz had proposed to measure the departure of a graph $G$ from regularity by the difference of the (adjacency) spectral radius and the average degree: $\epsilon(G)=\rho(G)-\frac{2m}{n}$. We give here new lower bounds on this quantity, which improve upon the currently known ones.
\end{abstract}

\subjclass{05C07,05C50,15A42,26D20,26D15}

\keywords{irregularity, adjacency matrix,  average degree,  spectral radius}

\thanks{{This research was supported by the Israel Science Foundation (grant number 862/10.)}}

\maketitle

\section{Introduction}\label{sec:intro}
\subsection{Motivation}

Let $G$ be a graph that has $n$ vertices and $m$ edges. The \emph{average degree} is $\overline{d}=\frac{2m}{n}$. Suppose now that $G$ has adjacency matrix $A$ and let us denote its spectral radius (\emph{i.e.} the largest modulus of an eigenvalue) by $\rho$. A classic 1957 result of Collatz and Sinogowitz \cite{ColSin57} is:
\begin{thm}\cite{ColSin57}\label{thm:cs}
Let $G$ be a graph with average degree $\overline{d}$ and spectral radius $\rho$. Then 
$$
\rho \geq \overline{d}
$$
and equality holds if and only if $G$ is regular.
\end{thm}

Theorem \ref{thm:cs} has served as the departure point for several interesting inquiries. As one particularly impressive recent example we may mention the independent discovery by Babai and Guiduli \cite{BabGui09} and by Nikiforov \cite{Nik10} of a spectral counterpart to the classic K\H{o}vari-S\'{o}s-Tur\'{a}n \cite{KovSosTur54} bound for the Zarankiewicz problem.

Another point of view inspired by Theorem \ref{thm:cs}
is to consider the difference $$\epsilon(G)=\rho-\overline{d}$$ as a measure for the irregularity of the graph $G$. This irregularity measure has been studied by various authors \cite{Extremal16,Bel92,CioGre07,Gol14collatz,Nik06,Nik07}.

\subsection{A brief digression about irregularity measures}
The simplest irregularity measure is that provided by the difference of the maximum and minimum degree (denoted, by $\Delta$ and $\delta$, respectively):
$$\Delta-\delta.$$
Though very simply defined and thus perhaps considered by some as too crude to be of use,this measure is actually quite useful in some contexts (cf. \cite{Yus13} for an example).  

Let us now introduce yet another irregularity measure, the variance of degrees:
$$
\var(G)=\frac{1}{n}\sum_{u \in V(G)}{\Big(d_{u}-\frac{2m}{n}\Big)^{2}}.
$$
Bell \cite{Bel92} compares $\epsilon(G)$ and $\var(G)$ for various classes of graphs. 

We wish to remark that the following relationship between $\Delta-\delta$ and $\var(G)$ is easily established by applying inequalities due to Popoviciu (cf. \cite[(1.4)]{ShaGupKap10}) and Nagy (cf. \cite[(1.5)]{ShaGupKap10}):
\begin{equation}\label{eq:var}
\frac{(\Delta-\delta)^{2}}{2n} \leq \var(G) \leq \frac{(\Delta-\delta)^{2}}{4}.
\end{equation}
The upper bound in \eqref{eq:var} has also been observed in \cite[p. 62]{ElpWoc13}.


For more alternative notions of graph irregularity we refer the interested reader to \cite{Irreg87,Alb97,ElpWoc13}.

\subsection{Main result}
Our purpose in this paper is to improve the extant lower bounds for $\epsilon(G)$, using rather elementary methods. The best bound to be found in the literature is due to Nikiforov \cite{Nik06}:
\begin{thm}\cite{Nik06}\label{thm:nik}
For every graph $G$,
\begin{equation}\label{eq:nik}
\epsilon(G) \geq \frac{\var(G)}{\sqrt{8m}}.
\end{equation}
\end{thm}

For example, as can be easily asscertained using \eqref{eq:var}, it implies the following bound obtained by Cioab\u{a} and Gregory in \cite{CioGre07}:
\begin{cor}\cite[Corollary 3]{CioGre07}
For every graph $G$,
$$
\epsilon(G) \geq \frac{(\Delta-\delta)^{2}}{4n\Delta}.
$$
\end{cor}

We shall prove, using elementary methods, the following new bound:
\begin{thm}\label{thm:main}
For every graph $G$,
\begin{equation}\label{eq:main}
\epsilon(G) \geq \frac{\var(G)\sqrt{n}}{\sqrt{8m\Delta}}.
\end{equation}
\end{thm}

As $n>\Delta$, the new bound of \eqref{eq:main} is always strictly better than \eqref{eq:nik}.

\section{Subregular graphs}

There is one very special case which merits separate treatment.
\begin{defin}\cite{Nik07}
Let $G$ be a graph with $\Delta-\delta=1$. If there is either exactly one vertex of degree $\Delta$ or exactly one vertex of degree $\Delta-1$, then $G$ is called \emph{subregular}.
\end{defin}

Clearly, subregular graphs are very close to being regular. We will find it convenient to distinguish between their two varieties thus:
\begin{defin}
Let $G$ be a subregular graph. 
\begin{itemize}[leftmargin=*]
\item
If there is exactly one vertex of degree $\Delta$, $G$ is \emph{high subregular}. 
\item
If there is exactly one vertex of degree $\Delta-1$, $G$ is \emph{low subregular}.
\end{itemize}
\end{defin}

For subregular graphs the bounds discussed so far yield estimates which are far too pessimistic. However, there is another bound due to Cioab\u{a} and Gregory \cite{CioGre07} which performs better in this case.

\begin{thm}\cite{CioGre07}\label{thm:cg1}
\begin{equation}\label{eq:cgs}
\epsilon(G) \geq \frac{1}{n(\Delta+2)}.
\end{equation}
\end{thm}

We will prove:
\begin{thm}\label{thm:subreg}
Let $G$ be a connected subregular graph on $n \geq 7$ vertices and with maximum degree $\Delta$. Then:
\begin{itemize}
\item
If $G$ is high subregular, then:
\begin{equation}\label{eq:subh}
\epsilon(G) \geq \frac{n^{2}-2n+3}{n^{3}\Delta}.
\end{equation}
\item
If $G$ is low subregular, then:
\begin{equation}\label{eq:subl}
\epsilon(G) \geq \frac{2n^{2}-4n-3}{2n^{3}(\Delta-1+\frac{1}{\Delta})}.
\end{equation}
\end{itemize}
\end{thm}


\begin{expl}
Consider the high subregular graph $G$ depicted in Figure \ref{fig:s1}. We have the following lower bounds for $\epsilon(G)$:

\begin{figure}
\includegraphics[height=6cm,width=8cm]{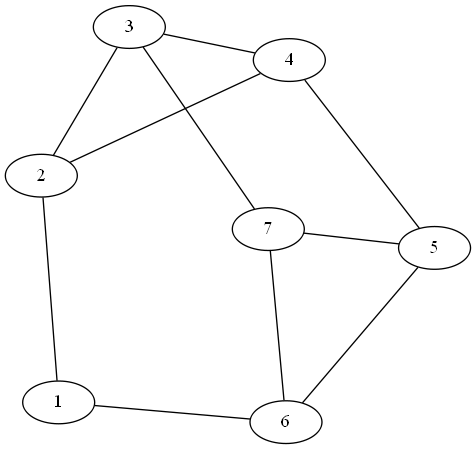}
\caption{A subregular graph, $n=7,\Delta=3$}\label{fig:s1}
\end{figure}

\medskip

\begin{tabular}{|l|llll|}

\hline

$\epsilon(G)$ \ & \eqref{eq:nik} \ & \eqref{eq:main} \ & \eqref{eq:cgs} \ & \eqref{eq:subh} \\ [0.5ex] 

\hline

 0.0461 &  0.0137 & 0.0209 & 0.0286 & 0.0364\\

\hline
\end{tabular}

\end{expl}

\section{Proof of Therem \ref{thm:main}}\label{sec:proof}

We begin by collecting a number of lemmae.
\begin{lem}[Hofmeister \cite{Hof88}]\label{lem:hof}
$$
\rho \geq \sqrt{\frac{1}{n}\sum_{i=1}^{n}{d_{i}^{2}}}.
$$
\end{lem}

\begin{lem}\label{lem:cs}
$$
\sqrt{\frac{1}{n}\sum_{i=1}^{n}{d_{i}^{2}}} \geq \frac{2m}{n}.
$$
\end{lem}
\begin{proof}
Cauchy-Shwarz.
\end{proof}

\begin{lem}\cite[p. 352]{Nik06}\label{lem:nik}
$$
\frac{1}{n}\sum_{i=1}^{n}{d_{i}^{2}}-\Big(\frac{2m}{n}\Big)^{2}=\var(G).
$$
\end{lem}

Let $A(G)$ and $D(G)$ be the adjacency matrix and the diagonal matrix of vertex degrees, respectively, of $G$. Then $Q(G)=A(G)+D(G)$ is called the \emph{signless Laplacian matrix}. The following claim is stated by Liu and Liu \cite{LiuLiu09zagr} only for connected graphs but in fact their proof does not use the connectedness assumption.

\begin{lem}\cite[Theorem 2.1]{LiuLiu09zagr}
Let $G$ be a graph. If $\rho$ is the spectral radius of $Q(G)$, then 
$$\sum_{i=1}^{n}{d_{i}^{2}} \leq m\rho.$$
\end{lem}

\begin{lem}\cite[Lemma 2.4]{LiuLiu09zagr}
Let $G$ be a graph. If $\rho$ is the spectral radius of $Q(G)$, then 
$$\rho \leq 2\Delta.$$
\end{lem}

We can now easily deduce:

\begin{lem}\label{lem:liuliu}
Let $G$ be a graph. Then 
$$\sum_{i=1}^{n}{d_{i}^{2}} \leq 2m\Delta.$$
\end{lem}

\begin{proof}[Proof of Theorem \ref{thm:main}]
First of all, in light of Lemma \ref{lem:hof} and \ref{lem:cs} we have: 
$$
\epsilon(G)=\rho(G)-\frac{2m}{n}\geq\sqrt{\frac{1}{n}\sum_{i=1}^{n}{d_{i}^{2}}} - \frac{2m}{n}=\frac{\frac{1}{n}\sum_{i=1}^{n}{d_{i}^{2}}-(\frac{2m}{n})^{2}}{\sqrt{\frac{1}{n}\sum_{i=1}^{n}{d_{i}^{2}}}+\frac{2m}{n}}.
$$
Now apply Lemma \ref{lem:nik} and then Lemma \ref{lem:cs} once more:
$$
\epsilon(G) \geq \frac{\frac{1}{n}\sum_{i=1}^{n}{d_{i}^{2}}-(\frac{2m}{n})^{2}}{\sqrt{\frac{1}{n}\sum_{i=1}^{n}{d_{i}^{2}}}+\frac{2m}{n}} = \frac{\var(G)}{\sqrt{\frac{1}{n}\sum_{i=1}^{n}{d_{i}^{2}}}+\frac{2m}{n}} \geq \frac{\var(G)}{2\sqrt{\frac{1}{n}\sum_{i=1}^{n}{d_{i}^{2}}}}.
$$
Finally, use Lemma \ref{lem:liuliu}:
$$
\epsilon(G) \geq \frac{\var(G)}{2\sqrt{\frac{1}{n}\sum_{i=1}^{n}{d_{i}^{2}}}} \geq \frac{\var(G)}{2\sqrt{\frac{2m\Delta}{n}}}=\frac{\var(G)\sqrt{n}}{\sqrt{8m\Delta}}.
$$

\end{proof}

\section{Proof of Theorem \ref{thm:subreg}}\label{sec:sub}

Our approach will be similar to that taken in the proof of Theorem \ref{thm:main} but instead of Hofmeister's bound for $\rho$ we shall need a more powerful one, due to Yu, Lu, and Tian \cite{YuLuTian04}. To state it, we define the \emph{$2$-degree} $t_{i}$ of the vertex $v_{i}$ as the sum of the degress of the vertices adjacent to $v_{i}$. That is:
$$
t_{i}=\sum_{j \sim i}{d_{j}}.
$$

\begin{lem}\cite{YuLuTian04}\label{lem:yulu}
Let $G$ be a connected graph. Then,
$$
\rho \geq \sqrt{\frac{\sum_{i=1}^{n}{t_{i}^{2}}}{\sum_{i=1}^{n}{d_{i}^{2}}}} \geq \overline{d}.
$$\end{lem}

\begin{lem}\label{lem:high}
Let $G$ be a high subregular graph on $n$ vertices and with maximum degree $\Delta$. Then $\Delta \leq n-2$.
\end{lem}
\begin{proof}
Suppose that $\Delta=n-1$. Then we have that all vertices but one are of degree $n-1$. But this means that the remaining vertex has $n-1$ neigbours as well. This is a contradiction.
\end{proof}

\begin{lem}\cite{HongShuFang01}\label{lem:hsf}
Let $G$ be a connected graph on $n$ vertices and $m$ edges, with maximum degree $\Delta$ and minimum degree $\delta$. Then,
$$
\rho \leq \frac{\delta-1+\sqrt{(\delta+1)^{2}+4(2m-\delta n)}}{2}.
$$\end{lem}
\begin{cor}\label{cor:hsf}
Let $G$ be a connected low subregular graph with maximum degree $\Delta$. Then 
$$
\rho \leq \Delta-1+\frac{1}{\Delta}.
$$
\end{cor}
\begin{proof}
By Lemma \ref{lem:hsf} we have
$$
\rho \leq \frac{\Delta-2+\sqrt{\Delta^{2}+4}}{2}.
$$
Our conclusion follows by observing that $\Delta^{2}+4 \leq (\Delta+\frac{2}{\Delta})^{2}$.
\end{proof}

\begin{proof}[Proof of Theorem \ref{thm:subreg}]\mbox{}

\textbf{Case:} \emph{$G$ is high subregular}

Let $v_{1}$ be the single vertex of degree $\Delta-1$ and let $v_{2},\ldots,v_{\Delta}$ be its neighbours. Then we have:
\begin{flalign*}
     &t_{1}=\Delta(\Delta-1),  \\
     &t_{i}=\Delta^{2}-1, \quad \quad 2 \leq i \leq \Delta, \\ 
     &t_{i}=\Delta^{2}, \quad \quad  \quad  \Delta +1 \leq i \leq n. \\
\end{flalign*}

Applying Lemma \ref{lem:yulu} we get:
\begin{flalign*}
\rho &\geq \sqrt{\frac{\Delta^2(\Delta-1)^2+(\Delta-1)(\Delta^2-1)^2+(n-\Delta)\Delta^4}{(\Delta-1)^2+(n-1)\Delta^2}}=\\
&=\sqrt{\frac{n\Delta^4 - 4\Delta^3 + 3\Delta^2 + \Delta - 1}{n\Delta^2 - 2\Delta + 1}}.
\end{flalign*}

The average degree in this case is:
$$
\overline{d}=\Delta-\frac{1}{n}.
$$

Consider now the following quantity:
$$
L(n,\Delta)=\frac{n\Delta^4 - 4\Delta^3 + 3\Delta^2 + \Delta - 1}{n\Delta^2 - 2\Delta + 1}-\Big(\Delta-\frac{1}{n}\Big)^{2}.
$$
Algebraic manipulation yields:
$$
L(n,\Delta)=\frac{(2\Delta^2 + \Delta - 1)n^2 + (2\Delta - 5\Delta^2)n + 2\Delta - 1}{n^2(n\Delta^2 - 2\Delta + 1)}.
$$
This expression is hardly manageable, but it simplifies dramatically upon observing that $L(n,\Delta)$ is a non-increasing function of $\Delta$ (this is verified by taking the partial derivative with respect to $\Delta$, we omit the simple but tedious details). Therefore, using Lemma \ref{lem:high} we have:
$$
L(n,\Delta) \geq L(n,n-2)=\frac{2n^4 - 12n^3 + 27n^2 - 22n - 5}{n^5 - 4n^4 + 2n^3 + 5n^2} \geq \frac{1}{n^2}(2n-4+\frac{6}{n}).
$$
Now we can complete the argument, using the well-known fact that $\Delta \geq \rho$:
\begin{flalign*}
\epsilon(G) &=\rho-\overline{d} = \frac{\rho^{2}-\overline{d}^{2}}{\rho+\overline{d}}  \geq \frac{L(n,\Delta)}{\rho+\overline{d}} \geq \frac{\frac{1}{n^2}(2n-4+\frac{6}{n})}{2\Delta}
=\frac{n^{2}-2n+3}{n^{3}\Delta}.
\end{flalign*}

\textbf{Case:} \emph{$G$ is low subregular}

As before, let $v_{1}$ be the single vertex of degree $\Delta$. We have:
\begin{flalign*}
     &t_{1}=\Delta(\Delta-1),  \\
     &t_{i}=\Delta^{2}-2\Delta+2, \quad \quad 2 \leq i \leq \Delta+1, \\ 
     &t_{i}=(\Delta-1)^{2}, \quad \quad  \quad  \quad \Delta +2 \leq i \leq n. \\
\end{flalign*}
Thus
$$
\rho \geq \sqrt{\frac{\Delta^{2}(\Delta-1)^{2}+\Delta(\Delta^{2}-2\Delta+2)^{2}+(n-\Delta-1)(\Delta-1)^{4}}{\Delta^{2}+(\Delta-1)^{2}(n-1)}}=
$$
$$
=\sqrt{\frac{n\Delta^{4}-(4n-4)\Delta^{3}+(6n-9)\Delta^{2}-(4n-7)\Delta+n-1}{n\Delta^{2}-(2n-2)\Delta+n-1}}.
$$
Keeping in mind that $\overline{d}=\Delta-1+\frac{1}{n}$ we define $L(n,\Delta)$ to be:	
$$
L(n,\Delta)=\frac{n\Delta^{4}-(4n-4)\Delta^{3}+(6n-9)\Delta^{2}-(4n-7)\Delta+n-1}{n\Delta^{2}-(2n-2)\Delta+n-1}-
$$
$$
-\Big(\Delta-1+\frac{1}{n}\Big)^{2}.
$$
After simplification we get:
$$
L(n,\Delta)=\frac{(2 \Delta^2 - 3 \Delta + 2) n^2 - (5 \Delta^2 - 8 \Delta + 3) n - 2 \Delta + 1}{(\Delta^2 - 2 \Delta + 1) n^3 + (2 \Delta - 1) n^2}.
$$
This function is also non-increasing with respect to $\Delta$ and thus we have:
$$
L(n,\Delta) \geq L(n,n-1)=\frac{2n^3 - 10n^2 + 15n - 3}{n^2(n^2 - 3n + 3)} \geq \frac{1}{n^{2}}(2n-4-\frac{3}{n}).
$$
To complete the argument we resort to Corollary \ref{cor:hsf}:
\begin{flalign*}
\epsilon(G) &=\rho-\overline{d} = \frac{\rho^{2}-\overline{d}^{2}}{\rho+\overline{d}}  \geq \frac{L(n,\Delta)}{\rho+\overline{d}} \geq \frac{\frac{1}{n^2}(2n-4-\frac{3}{n})}{2\rho}
=\frac{2n^{2}-4n-3}{2n^{3}(\Delta-1+\frac{1}{\Delta})}.
\end{flalign*}\end{proof}

\section{addendum}
Hong \cite{Hon93} raises the following problem (Problem 3 in his list):
\begin{qstn}
Let $G$ be the graph with the smallest value of $\epsilon(G)$ among non-regular graphs with $n$ vertices and $m$ edges. Is is true that $\Delta(G)-\delta(G)=1$?
\end{qstn}

We remark that Bell \cite{Bel92} has solved the problem of determining the graph with $n$ vertices and $m$ edges that has \emph{maximal} $\epsilon(G)$.

\bibliographystyle{abbrv}
\bibliography{nuim}
\end{document}